\documentclass[12pt,oneside,english]{amsart}
\usepackage[T1]{fontenc}
\usepackage[latin9]{inputenc}
\usepackage{mathrsfs}
\usepackage{amsthm}
\usepackage{amstext}
\usepackage{amssymb}

\makeatletter
\numberwithin{equation}{section}
\numberwithin{figure}{section}
\theoremstyle{plain}
\newtheorem{thm}{\protect\theoremname}[section]
  \theoremstyle{plain}
  \newtheorem{lem}[thm]{\protect\lemmaname}
  \theoremstyle{plain}
  \newtheorem{cor}[thm]{\protect\corollaryname}

\makeatother

\usepackage{babel}
  \providecommand{\corollaryname}{Corollary}
  \providecommand{\lemmaname}{Lemma}
\providecommand{\theoremname}{Theorem}

\begin{document}

\title[On Generalized Ordered Sets]{{\normalsize{}On Generalized Ordered Sets: a constructive development}}

\author{Jean S. Joseph}
\begin{abstract}
{\normalsize{}We propose a notion of a generalized order, which can
be used for the notion of a strict partial order. We introduce a weak
order to replace the usual weak order defined from a strict partial
order. In a constructive setting, that usual weak order causes problems
on the real numbers because their strict order cannot be proved to
be trichotomous.}{\normalsize \par}
\end{abstract}

\maketitle
\tableofcontents{}

\email{jsjean00@gmail.com}

\section{Introduction}

We introduce a structure that can be used in lieu of a strict partially
ordered set. In a constructive setting, the typical weak order on
a strict partially ordered set, which is defined as $x\leq y$ if
$x<y$ or $x=y$, is problematic on the real numbers (see 3.1). So
we propose another weak order, which we denote as $\leq_{P}$ (see
Section 3). We call such structure a \textit{generalized ordered set}:
it is a set with a binary relation satisfying three properties, namely
asymmetry%
\footnote{See Section3.%
}, transitivity%
\footnote{See Section 3.%
}, and positive antisymmetry%
\footnote{See Section 3.%
}. The first weak order $\leq$ has the advantage that it is an automatic
partial order, but our weak order $\leq_{P}$ lacks this advantage.
That is why we impose \textit{positive antisymmetry}. A generalized
ordered set is not a linearly ordered set, in that there are generalized
ordered sets that are not linearly ordered (see Section 3). More can
be found in subsequent sections, but we will summarize the main results.
In the third section, we introduce two weak orders on any set with
a relation $<$; one has a negative sense and the other a positive
sense. We have arrived at the following:
\begin{itemize}
\item when the binary relation $<$ on any set is asymmetric and cotransitive,
the two weak orders coincide (Theorems \ref{thm: 1} \& \ref{thm: 3}),
but the converse has no constructive proof (Theorem \ref{thm: 4});
\end{itemize}
later on that section, we show there are propositions about the real
numbers that have no constructive proofs if the weak order on any
set is taken to be $x\leq y$ if $x<y$ or $x=y$. We show the following:
\begin{itemize}
\item the statement ``for all $x,y\in\mathbb{R}$, if $\neg y<x$, then
$x<y$ or $x=y$'' implies the limited principle of omniscience (LPO)
(Theorem \ref{thm: 12}).
\end{itemize}
In the fourth section, we explain how any generalized order yields
a partial order, in the usual sense, and we show that a partial order
$\sim$ cannot automatically yield a generalized order with the definition
``$x\sim y$ and $\neg x=y$'' because
\begin{itemize}
\item a binary relation defined as such is not positively antisymmetric
(Section 4).
\end{itemize}
In the fifth section, we have arrived at a somewhat surprising result.
If we define a linearly ordered set to be a set with an asymmetric,
cotransitive%
\footnote{See 2.3.%
}, and negatively antisymmetric%
\footnote{See 2.4.%
} binary relation, then there is no constructive proof that the lexicographic
order turns the Cartesian product into a linearly ordered set. In
fact, we prove 
\begin{itemize}
\item the statement ``the lexicographic order on the cartesian product
of any two ordered sets is cotransitive'' implies the law of excluded
middle (Theorem \ref{thm: 15}).
\end{itemize}
However, the category of generalized ordered sets is closed under
Cartesian product (Theorem \ref{thm: 13}), which entails the following:
\begin{itemize}
\item the category of generalized ordered sets has all finite products (Corollary
\ref{cor: 15}).
\end{itemize}

\section{Preliminaries}

\subsection{Constructive Mathematics}

The theoretical development that we present below is constructive,
meaning the proofs of the theorems will omit the use of the law of
excluded middle, which is all proposition is either true or false.
Some of our theorems are ``metatheorems'', meaning they are theorems
about the theory we are developing; an example is Theorem \ref{thm: 15}.
The proof of these metatheorems will also be constructive. We need
to emphasize, though, that our focus will be more on the theory than
on the metatheory.

There is a vast literature on constructive mathematics, so we cannot
possibly be exhaustive about it here. More can be found in \cite{key-1,key-19,key-2-1,key-4,key-6}

\subsection{Omniscience Principle}

In constructive mathematics, there are certain statements that are
thought to have no proof. Some of these statements are the law of
excluded middle (for all proposition $P$, $P$ or $\neg P$), the
limited principle of omniscience (each binary sequence has all its
terms equal to zero or has a term equal to $1$, Markov's principle
(a binary sequence, whose all terms cannot be equal to zero, contains
a term equal to $1$), and the law of double negation (for all proposition
$P$, if $\neg\neg P$, then $P$). The list is longer than what is
given here; a more detailed rendering can be found in \cite{key-14,key-4,key-9-1}.
A use of an omniscience principle is to show that other statements
cannot have a constructive proof. For instance, if one needs to show
that a statement $T$ cannot be proved, then one can show that $T$
implies one of the omniscience principles. We have used such move
in the sections below; for instance, see Theorem \ref{thm: 4}.

\subsection{Cotransitivity}

A binary relation $<$ on a set $X$ is cotransitive if, for all $x,y,z\in X$,
$x<y$ implies $x<z$ or $z<y$ \cite{key-14}. Bridges and V{\^i}{\c t}{\u a}
in \cite{key-5} constructs a version of the real numbers whose strict
order is cotransitive {[}Proposition 2.1.8{]}. Also, Mandelkern in
\cite{key-12} defines a projective plane to have inequality relations
that are cotransitive. Any transitive, trichotomous relation on a
set is cotransitive; in particular, the strict orders on the integers,
the natural numbers, and the rational numbers are cotransitive.

\subsection{(Linearly) Ordered Sets}

Examples of generalized ordered sets are ordered sets (see Theorem
\ref{thm: 6}). Ordered sets are defined in \cite{key-11}. For completeness,
we include the definition here. An \textit{ordered set} is a set $X$
with a binary relation $<$ such that, for all $x,y,z\in X$, (1)
$x<y$ implies $y<x$ is false ; (2) $x<y$ implies $x<z$ or $z<x$;
(3) $x<y$ is false and $y<x$ is false imply $x=y$ (\textit{negative
antisymmetry}). The real numbers are ordered, so are all of its subsets.
In particular, the natural numbers, the integers, and the rational
numbers are ordered. 

An ordered set is a substitute for a linearly ordered set, defined
as a set with a binary relation $<$ such that, for all $x,y,z$,
(1) exactly one the following holds: $x<y$, $y<x$, or $x=y$, and
(2) $x<y<z$ implies $x<z$. The first condition is known as trichotomy.
Constructively, the usual notion of a linear order does not stand.
For instance, there is no constructive proof that the relation $<$
on the real numbers is trichotomous (see Theorem \ref{thm: 15-1}).

\section{Generalized Order}

Let $X$ be a set with a binary relation $<$. For $x,y\in X$, we
write $x\leq_{P}y$ if, for all $z\in X$, $z<x$ implies $z<y$,
and $y<z$ implies $x<z$. A \textit{generalized ordered set} is a
set $X$ with a binary relation $<$ such that, for all $x,y,z\in X$, 
\begin{itemize}
\item $x<y$ implies $y<x$ is false; (Asymmetry)
\item $x<y<z$ implies $x<z$; (Transitivity)
\item $x\leq_{P}y$ and $y\leq_{P}x$ imply $x=y$. (Positive Antisymmetry)
\end{itemize}
The \textit{dual} of a generalized ordered set $X$ is the generalized
ordered set $X_{d}$ whose underlying set is $X$ and binary relation
is $x<y$ in $X_{d}$ if and only if $y<x$ in $X$.
\begin{lem}
Let $X$ be a generalized ordered set. Then $x\leq_{P}y$ in $X$
if and only if $y\leq_{P}x$ in $X_{d}$.\end{lem}
\begin{proof}
Suppose $x\leq_{P}y$ in $X$. For all $z\in X_{d}$, if $z<y$ in
$X_{d}$, then $y<z$ in $X$, so $x<z$, implying $z<x$ in $X_{d}$.
Similarly, $x<z$ in $X_{d}$ implies $y<z$ in $X_{d}$. Hence, $y\leq_{P}x$
in $X_{d}$. The converse is obtained in a similar way.
\end{proof}
In what follows, we show that any ordered set is a generalized ordered
set. Recall that we briefly talk about ordered sets in the previous
section, but a more detailed account on ordered sets is in \cite{key-11}.
\begin{lem}
Let $X$ be a generalized ordered set: \label{lem: 1}
\begin{enumerate}
\item $\leq_{P}$ is transitive;
\item $x<y$ implies $x\leq_{P}y$.
\end{enumerate}
\end{lem}
\begin{proof}
(1) Let $x,y,z\in X$. Suppose $x\leq_{P}y\leq_{P}z$. Let $u\in X$.
If $u<x$, then $u<y$ since $x\leq_{P}y$, so $u<z$ since $y\leq_{P}z$.
Similarly, $z<u$ implies $x<u$. Therefore, $x\leq_{P}z$.

(2) If $x<y$, then, for all $z\in X$, $z<x$ implies $z<y$ and
$y<z$ implies $x<z$, by transitivity.
\end{proof}
We write $x\leq_{N}y$ for $\neg y<x$.
\begin{thm}
Let $X$ be a set with a cotransitive binary relation $<$. If $x\leq_{N}y$
then $x\leq_{P}y$. \label{thm: 1}\end{thm}
\begin{proof}
For all $z\in X$, suppose $z<x$. Since $<$ is cotransitive, either
$z<y$ or $y<x$. But $y<x$ is false because $x\leq_{N}y$, so $z<y$.
Similarly, $y<z$ implies $x<z$. Hence, $x\leq_{P}y$.\end{proof}
\begin{thm}
The statement ``For any set $X$ with a binary relation $<$, if,
for all $x,y\in X$, $x\leq_{N}y$ implies $x\leq_{P}y$, then that
binary relation is cotransitive'' implies the weak excluded middle\textup{}%
\footnote{Weak excluded middle is, for all proposition $P$, either $\neg P$
or $\neg\neg P$%
}\textup{. \label{thm: 4}}\end{thm}
\begin{proof}
Let $P$ be any proposition. Let $X=\left\{ a,b,c\right\} $ with
$a=a$, $b=b$, $c=c$ and with exactly $a<b$, $a<c$ if $\neg\neg P$,
and $c<b$ if $\neg P$. Suppose $x\leq_{N}y$. The two cases worth
mentioning are when $x=b$,$y=c$ and when $x=c$,$y=a$. 

Suppose $b\leq_{N}c$. Let $z\in X$. Also suppose $z<b$. Note that
$b\leq_{N}c$ entails $\neg\neg P$ by definition, so $a<c$ by definition;
hence, when $z=a$, we have $a<b$ implies $a<c$. Note that when
$z=b,c$, ``$z<b$ implies $z<c$'' holds trivially. Now, suppose
$c<z$. Note also that ``$c<z$ implies $b<z$'' holds trivially
for all $z$. Therefore, $b\leq_{P}c$.

Suppose $c\leq_{N}a$. Let $z\in X$. Note that ``$z<c$ implies
$z<a$'' holds trivially. Now, suppose $a<z$. Note that $c\leq_{N}a$
entails $\neg\neg\neg P$, so $\neg P$, implying $c<b$ by definition;
hence, when $z=b$, we have $a<b$ implies $c<b$. Also, note that
``$a<z$ implies $c<z$'' holds trivially, when $z=a,c$. Thus,
$c\leq_{P}a$.

Therefore, the binary relation $<$ on $X$ is cotransitive. Since
$a<b$ by definition, either $a<c$, in which case $\neg\neg P$,
or $c<b$, in which case $\neg P$.\end{proof}
\begin{thm}
Let $X$ be a set with an asymmetric binary relation $<$. If $x\leq_{P}y$
then $x\leq_{N}y$. \label{thm: 3}\end{thm}
\begin{proof}
Suppose $x\leq_{P}y$. If $y<x$, then $y<y$, which is false by asymmetry.\end{proof}
\begin{cor}
Let $X$ be an ordered set as in \cite{key-11}. Then $x\leq_{N}y$
if and only if $x\leq_{P}y$. \label{cor: 4}\end{cor}
\begin{proof}
By Theorems \ref{thm: 1} and \ref{thm: 3}.\end{proof}
\begin{thm}
Any ordered set in \cite{key-11} is a generalized ordered set. \label{thm: 6}\end{thm}
\begin{proof}
Let $X$ be an ordered set and $x,y,z\in X$. Suppose $x<y<z$. By
cotransitivity, either $y<x$ or $x<z$. But $y<x$ is false by asymmetry,
so $x<z$. Hence, transitivity holds. 

Suppose $x\leq_{P}y$ and $y\leq_{P}x$. Then $x\leq_{N}y$ and $y\leq_{N}x$,
by Corollary \ref{cor: 4}, so $x=y$ by negative antisymmetry. Hence,
positive antisymmetry holds.
\end{proof}
For generalized ordered sets $X$ and $Y$, an \textit{embedding}
$f$ of $X$ into $Y$ is a function from $X$ to $Y$ such that $x<y$
if and only if $f\left(x\right)<f\left(y\right)$. We call $\mathbf{gOrd}$
the category of generalized ordered sets with their embeddings, and
we call $\mathbf{Ord}$ the category of ordered sets with their embeddings.
An embedding of ordered sets is defined as an embedding of generalized
ordered sets.
\begin{cor}
$\mathbf{Ord}$ is a subcategory of $\mathbf{gOrd}$.\end{cor}
\begin{proof}
By Theorem \ref{thm: 6}, the objects of $\mathbf{Ord}$ are objects
of $\mathbf{gOrd}$. The maps of $\mathbf{Ord}$ are those of $\mathbf{gOrd}$,
by definition.
\end{proof}
Not all generalized ordered sets are ordered sets. Let $X=\left\{ a,b,c\right\} $
with a binary relation defined as 
\[
\begin{array}{ccc}
 &  & b\\
 &  & \uparrow\\
c &  & a
\end{array}.
\]
Note that $X$ is a generalized ordered set, but its binary relation
is not cotransitive because $a<b$ but neither $a<c$ nor $c<b$.
We give an example of a generalized ordered set, whose binary relation
$<$ cannot be proved to be cotransitive. Let $\mathsf{Prop}$ be
the collection of all propositions with equality defined as $P=Q$
if $P\iff Q$ and with a binary relation defined as $P<Q$ if $\neg P\wedge Q$.
The proposition ``True'' is denoted by $1$ and the proposition
``False'' by $0$. Observe that $0<1$. 
\begin{thm}
$\mathsf{Prop}$ is a generalized ordered set.\end{thm}
\begin{proof}
If $P<Q$, then $\neg P\wedge Q$, and if $Q<P$, then $\neg Q\wedge P$.
But $\neg P\wedge P$ is false. Hence, asymmetry holds.

Note that $P<Q<R$ implies $P<R$ is vacuously true because $P<Q<R$
is always false. Hence, transitivity holds.

Suppose $P\leq_{P}Q$. If $P$, then $0<P$, so $0<Q$, implying $Q$.
Hence, $P\implies Q$. Similarly, $Q\leq_{P}P$ implies $Q\implies P$.
Therefore, $P=Q$. Hence, positive antisymmetry holds.\end{proof}
\begin{thm}
The statement ``the binary relation $<$ on $\mathsf{Prop}$ is cotransitive''
implies the law of excluded middle. \end{thm}
\begin{proof}
Let $P$ be any proposition. Since $0<1$, either $0<P$, in which
case $P$, or $P<1$, in which case $\neg P$.
\end{proof}

\subsection{Weak Order on the Real Numbers}

If we define the weak order on the real numbers as $x\leq y$ if $x<y$
or $x=y$, then some theorems for the real numbers, which are known
to be true, will turn out to have no constructive proofs. For instance,
the proposition 
\[
\textrm{for all }x,y\in\mathbb{R},\textrm{ if }\neg y<x\textrm{ then }x\leq y\textrm{ (*)}
\]
 is true for other definitions of $\leq$ on the real numbers. For
instance, when the real numbers are defined as the completion of the
rational numbers, such as in \cite{key-11}, the proposition ({*})
is trivially true since $x\leq y$ \textit{is }$\neg y<x$. The same
goes on as in \cite{key-6-1}. In \cite{key-5}, this proposition
({*}) is also true (see Lemma 2.1.4). But if the weak order is defined
as $x\leq y$ := $x<y$ or $x=y$, then the proposition ({*}) has
no constructive proof. A reason for this problem is that the relation
$<$ on the real numbers cannot be proved to be trichotomous, constructively.
Here are the details.
\begin{thm}
``The binary relation $<$ on $\mathbb{R}$ is trichotomous'' implies
the limited principle of omniscience (LPO), which says for all $\left(a_{n}\right):\mathbf{2}^{\mathbb{N}}$,
either for all $n$, $a_{n}=0$, or there is $n$ such that $a_{n}=1$.
\label{thm: 15-1}\end{thm}
\begin{proof}
Let $\left(a_{n}\right)$ be a decreasing%
\footnote{Note that if each decreasing binary sequence has all zero terms or
has a term equal to $1$, then LPO.%
} binary sequence. Let $x=\sum_{n=0}^{\infty}\frac{a_{n}}{2^{n}}$.
Note that $x$, as a geometric series%
\footnote{For more on geometric series, see \cite{key-3-1}.%
}, converges in $\mathbb{R}$. By assumption, $x<0$, $x>0$, or $x=0$.
If $x=0$, then $a_{0}=0$, so $a_{n}=0$ for all $n$. Since each
$\frac{a_{n}}{2^{n}}$ is nonnegative, each partial sum is nonnegative,
so $x<0$ is impossible. If $0<x$, there is a nonnegative integer
$N$ such that $0<\sum_{n=0}^{n=N}\frac{a_{n}}{2^{n}}$ , so there
is $k$ in $\left\{ 0,\ldots,N\right\} $ such that $a_{k}=1$. Therefore,
LPO.\end{proof}
\begin{thm}
``For all $x,y\in\mathbb{R}$, if $\neg y<x$, then $x<y$ or $x=y$''
implies the limited principle of omniscience (LPO).\textup{ \label{thm: 12}}\end{thm}
\begin{proof}
Let $\left(a_{n}\right)$ be any binary sequence, and let $x=\sum_{n=0}^{\infty}\frac{a_{n}}{2^{n}}$.
Since $x$ is a geometric series with a ratio less than $1$, it converges
in $\mathbb{R}$. Note that it is false $x<0$ since each $\frac{a_{n}}{2^{n}}$
is nonnegative, so either $0<x$ or $0=x$, by our assumption. If
$0<x$, then there is a nonnegative $N$ such that $0<\sum_{n=0}^{N}\frac{a_{n}}{2^{n}}$
, so there must be $k\in\left\{ 0,\ldots,N\right\} $ such that $a_{k}=1$.
If $0=x$, then $a_{n}=0$ for all $n$.
\end{proof}

\section{Connection with Partially Ordered Sets}

Generalized ordered sets are related to partially ordered sets. A
lot more can be found about partially ordered sets on the web and
in \cite{key-21}, but we provide the definition here: a set $X$
with a binary relation $\sim$ is \textit{partially ordered} if, for
all $x,y,z\in X$, (1) $x\sim x$; (2) $x\sim y\sim z$ implies $x\sim y$;
(3) $x\sim y$ and $y\sim x$ imply $x=y$.

Given a generalized ordered set $\left(X,<\right)$, one can create
a partially ordered set as follows. Keep the underlying set $X$ and
take the binary relation to be $\leq_{p}$, which is defined from
the previous section. That $x\leq_{p}x$ for all $x\in X$ is automatic.
Lemma \ref{lem: 1} (1) proves that, for all $x,y,z\in X$, $x\leq_{p}y\leq_{p}z$
implies $x\leq_{p}z$. Lastly, positive antisymmetry, as defined in
the previous section, is the third property.

Getting a generalized ordered set from a partially ordered set is
not immediate. Let $\left(X,\sim\right)$ be a partially ordered set.
We keep the underlying set, and we define the binary relation $<$
to be $x<y$ := $x\sim y$ and $\neg x=y$. We will prove a two-part
lemma for asymmetry and transitivity of $<$, and we will discuss
the problem with positive antisymmetry.
\begin{lem}
\textup{Let $\left(X,\sim\right)$ be a partially ordered set with
a binary relation $<$ defined as $x<y$ := $x\sim y$ and $\neg x=y$.
Then, for all $x,y,z\in X$, \label{lem: 11}}
\begin{enumerate}
\item $x<y$ implies $\neg y<x$;
\item $x<y<z$ implies $x<z$.
\end{enumerate}
\end{lem}
\begin{proof}
(1) If $x<y$ and $y<x$, then $x\sim y$ and $y\sim x$, so $x=y$,
by property (3) of a partially ordered set. This is impossible because
$\neg x=y$ from the definition of $x<y$.

(2) If $x<y<z$, then $x\sim y\sim z$, so $x\sim z$ from property
(2) of a partially ordered set. To prove $\neg x=z$. Assume $x=z$,
so $z\sim y$ because $x\sim y$ and $\sim$ is well defined. Since
$y\sim z$, we have $y=z$, which is impossible since $y<z$. Therefore,
$x<z$.
\end{proof}
We now show why positive antisymmetry cannot hold, given an arbitrary
partially ordered set. We start with the set $X=\left\{ a,b\right\} $
with equality defined as $a=a$ and $b=b$, and we impose $\neg a=b$.
The binary relation $\sim$ on $X$ is $a\sim a$ and $b\sim b$.
To see that $\left(X,\sim\right)$ is partially ordered is immediate.
We recall that the binary relation $<$ defined earlier is $x<y$
:= $x\sim y$ and $\neg x=y$; also we recall that the binary relation
$\leq_{p}$ is as defined in the third section. Note that $a\leq_{p}b$.
The reason is as follows. Let $c\in X$ and suppose $c<a$. By definition
of $<$, we have $c\sim a$ and $\neg c=a$. But $c\sim a$ implies
$c$ is $a$ by the definition of $\sim$ on $X$, so $\neg c=a$
must be false. Hence, the statement ``for all $c\in X$, $c<a$ implies
$c<b$'' is always true, vacuously. Similarly, the statement ``for
all $c\in X$, $b<c$ implies $a<c$'' is true. With a similar argument,
we can prove $b\leq_{p}a$. However, as we have defined $X$, it is
false that $a=b$. 

We need to emphasize that we have shown that an arbitrary partially
ordered set $\left(X,\sim\right)$ cannot be turned into a generalized
ordered set with a binary relation $<$ defined as $x<y$ := $x\sim y$
and $\neg x=y$. One question is whether there is any way to define
a generalized order on an arbitrary partially ordered set.

A partial answer is that the binary relation $<$ we define above
is positively antisymmetric for certain classes of partially ordered
sets. Namely they are the partially ordered sets $X$ such that, for
all $x,y\in X$, $\neg y<x$ implies $x\sim y$. In particular, partially
ordered sets with a total order and a decidable equality fall in that
class. A binary relation $\sim$ on set $X$ is \textit{total }if,
for all $x,y\in X$, $x\sim y$ or $y\sim x$, and the equality on
$X$ is \textit{decidable} if, for all $x,y\in X$, $x=y$ or $\neg x=y$.
We show the details with the following:
\begin{thm}
Let $X$ be a partially ordered set, and let the binary relation $<$
be defined as $x<y$ := $x\sim y$ and $\neg x=y$. In addition, let
$X$ satisfy the condition ({*}): for all $x,y\in X$, $\neg y<x$
implies $x\sim y$. Then \label{thm: 12-1}
\begin{enumerate}
\item $x\leq_{p}y$ implies $x\sim y$;
\item positive antisymmetry holds.
\end{enumerate}
\end{thm}
\begin{proof}
(1) It suffices to show $x\leq_{p}y$ implies $\neg y<x$. This follows
from Theorem \ref{thm: 3} and Lemma \ref{lem: 11}(1). Hence, $x\leq_{p}y$
implies $x\sim y$, by the condition ({*}).

(2) Suppose, for all $x,y\in X$, $x\leq_{p}y$ and $y\leq_{p}x$.
Then $x\sim y$ and $y\sim x$, by (1). Thus $x=y$.
\end{proof}
By the way, partially ordered sets with a total order and a decidable
equality are not exactly those in Theorem \ref{thm: 12-1}. Although
the former is always the latter, the converse cannot be proved constructively.
We again show the details with the following:
\begin{thm}
Let $X$ be a partially ordered set, and let the binary relation $<$
be defined as $x<y$ := $x\sim y$ and $\neg x=y$. Then, if $\sim$
is total and $=$ is decidable, then, for all $x,y\in X$, $\neg y<x$
implies $x\sim y$.\end{thm}
\begin{proof}
Suppose $\neg y<x$, meaning $\neg\left(y\sim x\textrm{ and }\neg y=x\right)$.
If $y=x$, then $x\sim y$ since $x\sim x$ and $\sim$ is well defined.
If $\neg y=x$, then $y\sim x$ must be false, so $x\sim y$ since
$\sim$ is total.\end{proof}
\begin{thm}
The statement ``For any partially ordered set $X$, if for all $x,y\in X$,
$\neg y<x$ implies $x\sim y$, then $\sim$ is total and $=$ is
decidable'' implies the weak excluded middle%
\footnote{Weak excluded middle is, for all proposition $P$, either $\neg P$
or $\neg\neg P$.%
}.\end{thm}
\begin{proof}
Let $P$ be any proposition. Let $X=\left\{ a,b\right\} $ be the
partially ordered set with equality defined as $a=a$ and $b=b$.
We also impose the condition $\neg b=a$. The partial order is defined
as $a\sim a$, $b\sim b$; $a\sim b$ exactly when $\neg P$, and
$b\sim a$ exactly when $\neg\neg P$. The statement ``$\neg y<x$
implies $x\sim y$'' holds trivially when $x,y=a$ and when $x,y=b$,
because $a\sim a$ and $b\sim b$. If $\neg a<b$, then $a\sim b$
must be false because $\neg a=b$, so $\neg\neg P$, implying $b\sim a$.
Finally, if $\neg b<a$, then $b\sim a$ must be false because $\neg b=a$,
so $\neg\neg\neg P$, implying $\neg P$; hence, $a\sim b$. Since
the statement ``for all $x,y\in X$, $\neg y<x$ implies $x\sim y$''
holds, the partial order $\sim$ is total, so $x\sim y$ or $y\sim x$,
implying $\neg P$ or $\neg\neg P$. 
\end{proof}

\section{Orders on Products}

\subsection{Coarse Product}

Given two sets $X,Y$, the cartesian product $X\times Y$ is the set
whose elements are pairs of elements of $X$ and $Y$ and whose equality
is defined as $\left(x,y\right)=\left(x',y'\right)$ if $x=x'$ and
$y=y'$. When $X,Y$ have each an order relation, a typical order
defined on $X\times Y$ is the \textit{lexicographic order}, which
is defined as $\left(x,y\right)<\left(x',y'\right)$ if $x<x'$, or
$x=x'$ and $y<y'$. When $X,Y$ are ordered sets, as defined in 2.4,
there is no constructive proof that $X\times Y$, with the lexicographic
order, is an ordered set. But there are some classes of cartesian
products that are ordered sets when the components are ordered sets.
For instance, we have shown in \cite{key-11} that the cartesian product
with the lexicographic order is an ordered set if the first component
satisfies the property: for all $x,y$, $x=y$, $x<y$, or $y<x$.
Here is why no constructive proof exists for the more general theorem:
\begin{thm}
The statement ``the lexicographic order on the cartesian product
of any two ordered sets is cotransitive'' implies the law of excluded
middle. \label{thm: 15}\end{thm}
\begin{proof}
Let $P$ be any proposition. Let $X$ be $\left\{ 0\right\} \cup\left\{ 1:P\right\} \cup\left\{ 1:\neg P\right\} \subseteq\left\{ 0,1\right\} \subseteq\mathbb{Z}$.
Equality on $X$ is $0=0$ and $1=1$, and the binary relation $<$
on $X$ is defined as $0<1$ if $P$ and $1<0$ if $\neg P$. Let
$Y$ be $\left\{ 0,1\right\} \subseteq\mathbb{Z}$; the binary relation
$<$ on $Y$ is inherited from $\mathbb{Z}$. Let $S$ be the subset
$\left\{ 0\in X:\neg P\right\} \cup\left\{ 1\in X:P\right\} $, so
$S\times Y\subseteq X\times Y$. The cartesian product $X\times Y$
has the lexicographic order, so $\left(0,0\right)<\left(0,1\right)$.
Let $\left(p,q\right)\in S\times Y$. By cotransitivity, either $\left(0,0\right)<\left(p,q\right)$
or $\left(p,q\right)<\left(0,1\right)$. We will refer to the latter
disjuncts as ``first case'' and ``second case''. If the first
case holds, then $0<p$, or $0=p$ and $0<q$. If $0<p$, then $p=1$,
so $P$; if $0=p$, then $0\in S$, so $\neg P$. If the second case
holds, then $p<0$, or $p=0$ and $q<1$. If $p<0$, then $\neg P$;
if $p=0$, then $0\in S$, so $\neg P$.
\end{proof}
A question is whether it is possible to define a cotransitive binary
relation on the cartesian product of any two ordered sets. A partial
answer is that we might need to consider other types of products.
For instance, given any two sets $X,Y$, we call the\textit{ coarse
product} of $X$ and $Y$ the set whose elements are pairs $\left(x,y\right)$,
where $x\in X$ and $y\in Y$, and whose equality is defined as $\left(x,y\right)=\left(x',y'\right)$
if $x=x'$; we also can define the equality as $\left(x,y\right)=\left(x',y'\right)$
if $y=y'$. If we refine the notation to differentiate this product
from the cartesian product, we denote by $X\times_{c}Y$ the coarse
product and by $=_{c,X}$ and $=_{c,Y}$ the first and second equalities,
respectively. 

For any given sets $X,Y$, there is a natural function from $X\times Y$
into $X\times_{c}Y$, namely the function $\left(x,y\right)\mapsto\left(x,y\right)$.
Just for the purpose of illustration, say $X=Y=\left\{ 0,1,2\right\} \subseteq\mathbb{Z}$.
The cartesian product of $X$ and $Y$ has $9$ distinct elements,
while the coarse product of $X$ and $Y$, with either equality, has
$3$ distinct elements. The coarse product of $X$ and $Y$ is an
ordered set if either $X$ is an ordered set or $Y$ is an ordered
set. The binary relation on the coarse product is defined as $\left(x,y\right)<_{X}\left(x',y'\right)$
if $x<x'$ or $\left(x,y\right)<_{Y}\left(x',y'\right)$ if $y<y'$.
As a summary, we have the following:
\begin{thm}
Let $X,Y$ be any sets: 
\begin{enumerate}
\item if $X$ is an ordered set, then $\left(X\times_{c}Y,=_{c,X}\right)$
is an ordered set with $\left(x,y\right)<_{X}\left(x',y'\right)$
if $x<x'$;
\item if $Y$ is an ordered set, then $\left(X\times_{c}Y,=_{c,Y}\right)$
is an ordered set with $\left(x,y\right)<_{Y}\left(x',y'\right)$
if $y<y'$.
\end{enumerate}
\end{thm}

\subsection{Finite and Infinite Products}

Contrary to the ordered sets, the lexicographic order on the cartesian
product of two generalized ordered sets is a generalized order. We
include the details below. In what follows, when we write $x\leq y$,
we mean $x\leq_{P}y$.
\begin{thm}
The cartesian product of two generalized ordered sets with the lexicographic
order is a generalized ordered set. \label{thm: 13}\end{thm}
\begin{proof}
Suppose $\left(x,y\right)<\left(x',y'\right)$. Then either $x<x'$,
or $x=x'$ and $y<y'$. If $x<x'$, then $\left(x',y'\right)<\left(x,y\right)$
is false because the binary relation on $X$ is asymmetric. If $x=x'$
and $y<y'$, then $\left(x',y'\right)<\left(x,y\right)$ is false
because the binary relations on $X$ and on $Y$ are asymmetric.

Transitivity of the lexicographic order follows from the well-definedness
and transitivity of the binary relation on $X$ and the transitivity
of the binary relation on $Y$. 

For positive antisymmetry, suppose the following: for all $\left(r,s\right),\left(u,v\right)\in X\times Y$,
\[
\left\{ \begin{array}{c}
\left(r,s\right)<\left(x,y\right)\textrm{ implies }\left(r,s\right)<\left(x',y'\right),\textrm{ and }\\
\left(x',y'\right)<\left(r,s\right)\textrm{ implies }\left(x,y\right)<\left(r,s\right)
\end{array}\right\} (1)
\]
\[
\left\{ \begin{array}{c}
\left(u,v\right)<\left(x',y'\right)\textrm{ implies }\left(u,v\right)<\left(x,y\right),\textrm{ and }\\
\left(x,y\right)<\left(u,v\right)\textrm{ implies }\left(x',y'\right)<\left(u,v\right)
\end{array}\right\} (2).
\]
 To show $x\leq x'$, suppose, for all $z\in X$, $z<x$. Then $\left(z,y'\right)<\left(x,y\right)$,
so $\left(z,y'\right)<\left(x',y'\right)$ by (1); hence, $z<x'$.
Now, if $x'<z$, then $\left(x',y'\right)<\left(z,y\right)$, so $\left(x,y\right)<\left(z,y\right)$
by (1), implying $x<z$. Similarly, $x'\leq x$. Thus $x=x'$. To
show $y\leq y'$, suppose, for all $z\in Y$, $z<y$. Then $\left(x,z\right)<\left(x,y\right)$,
so $\left(x,z\right)<\left(x',y'\right)$. Since $x=x'$, it follows
$z<y'$. Now, if $y'<z$, then $\left(x',y'\right)<\left(x',z\right)$,
so $\left(x,y\right)<\left(x',z\right)$, implying $y<z$ since $x=x'$.
Similarly, $y'\leq y$. Thus $y=y'$.
\end{proof}
For any nonnegative integer $n$, the \textit{lexicographic order}
on the cartesian product $\prod_{i=0}^{n}X_{i}$, where each $X_{i}$
is a generalized ordered set, is $\left(x_{0},\ldots,x_{n}\right)<\left(y_{0},\ldots,y_{n}\right)$
if there is $k\in\left\{ 0,\ldots,n\right\} $ such that $x_{k}<y_{k}$
and, for all $j<k$, $x_{j}=y_{j}$.
\begin{thm}
For each nonnegative integer $n$, the cartesian product $\Pi_{i=0}^{n}X_{i}$,
where each $X_{i}$ is a generalized ordered set, is a generalized
ordered set. \label{thm: 23}\end{thm}
\begin{proof}
The proof goes by induction on $n$ and then follows from Theorem
\ref{thm: 13}.
\end{proof}
In the language of category theory, we have the following:
\begin{cor}
The category%
\footnote{See Section 3 for the definition of the category of generalized ordered
sets.%
} of generalized ordered sets has all finite products. \label{cor: 15}\end{cor}
\begin{proof}
The proof is immediate from Theorem \ref{thm: 23}.
\end{proof}
Besides the set of all propositions $\mathsf{Prop}$ we presented
in the third section, we will exhibit a family of nontrivial generalized
ordered sets (see Corollary \ref{cor: 23}). First, we define the
cartesian product of an arbitrary collection of generalized ordered
sets indexed by elements of any ordinal. Instead of Cantor's ordinals
\cite{key-7}, we use ordinals%
\footnote{For more on ordinals, see \cite{key-17}%
} as defined in \cite{key-18}. 

For an ordinal $I$ as defined in \cite{key-18}, let $\mathfrak{g}$
be a function from $I$ to the class of generalized ordered sets,
and we write $X_{i}$ for $\mathfrak{g}\left(i\right)$. An example
of such a $\mathfrak{g}$ is the function that sends each $i\in I$
to $\mathbb{N}$. We write $\prod_{i,\mathfrak{g}}X_{i}$ for the
set of functions $f$ from $I$ to $\cup_{i}X_{i}$ such that $f\left(i\right)\in X_{i}$.
For instance, if each $X_{i}$ has a distinguished element, say $0_{X_{i}}$,
an element of $\prod_{i,\mathfrak{g}}X_{i}$ is the function that
sends each $i$ to $0_{X_{i}}$. If $\mathfrak{g}$ is the constant
function, say $\mathfrak{g}\left(i\right)=X$ for all $i$, then $\prod_{i,\mathfrak{g}}X$
is $X^{I}$, the set of functions from $I$ to $X$. If $A\subseteq I$,
we write $\prod_{i\in A,\mathfrak{g}}X_{i}$ for the set of functions
of $\prod_{i,\mathfrak{g}}X_{i}$ restricted to $A$, where $\mathfrak{g}$
is restricted to $A$. The \textit{lexicographic order} on $\prod_{i,\mathfrak{g}}X_{i}$
is $f<g$ if there is $k\in I$ such that $f\left(k\right)<g\left(k\right)$
and, for all $j<k$, $f\left(j\right)=g\left(j\right)$; in such case,
we say $k$ witnesses $f<g$. The equality on $\prod_{i,\mathfrak{g}}X_{i}$
is pointwise equality. For $A\subseteq I$, the \textit{lexicographic
order} on $\prod_{i\in A,\mathfrak{g}}X_{i}$ is $f<g$ if there is
$k\in A$ such that $f\left(k\right)<g\left(k\right)$ and, for all
$j\in A$ with $j<k$, $f\left(j\right)=g\left(j\right)$.

A \textit{lower subset} $S$ of a generalized ordered set is a set
such that $s'\leq s\in S$ implies $s'\in S$. For an ordinal $I$
as defined in \cite{key-18}, an \textit{initial segment} of $I$
is a lower subset of $I$. 
\begin{thm}
Let $\mathscr{L}$ be any collection of initial segments of $I$.
If, for each finite subcollection $\left\{ A_{0},\ldots,A_{m}\right\} $
of $\mathscr{L}$, $\prod_{i\in\cup_{j=1}^{m}A_{j},\mathfrak{g}}X_{i}$
is a generalized ordered set, then the lexicographic order on $\prod_{i\in\cup\mathscr{L},\mathfrak{g}}X_{i}$
is asymmetric and transitive. If, in addition, the binary relation
$<$ on each $X_{i}$ is cotransitive, then the lexicographic order
on $\prod_{i\in\cup\mathscr{L},\mathfrak{g}}X_{i}$ is positively
antisymmetric. \label{thm: 18}\end{thm}
\begin{proof}
For asymmetry, suppose $f<g$ and $g<f$ in $\prod_{i\in\cup\mathscr{L},\mathfrak{g}}X_{i}$.
Since $k$ witnesses $f<g$, $f\left(k\right)<g\left(k\right)$, and
since $k'$ witnesses $g<f$, $g\left(k'\right)<f\left(k'\right)$.
Since $k$ is in some $A\in\mathscr{L}$ and $k'$ in some $A'\in\mathscr{L}$,
it follows $\prod_{i\in A\cup A',\mathfrak{g}}X_{i}$ is a generalized
ordered set by supposition. Since $k,k'\in A\cup A'$, it follows
$f<g$ and $g<f$ in $\prod_{i\in A\cup A',\mathfrak{g}}X_{i}$, when
$f$ and $g$ are restricted to $A\cup A'$. But that is impossible
since the binary relation on $\prod_{i\in A\cup A',\mathfrak{g}}X_{i}$
is asymmetric.

For transitivity, suppose $f<g<h$ in $\prod_{i\in\cup\mathscr{L},\mathfrak{g}}X_{i}$.
Then $k$ witnesses $f<g$ and $r$ witnesses $g<h$, where $k$ is
in some $A$ and $r$ in some $A'$. Hence, $f<h$ in $\prod_{i\in A\cup A',\mathfrak{g}}X_{i}$.
Thus, $f<h$ in $\prod_{i\in\cup\mathscr{L},\mathfrak{g}}X_{i}$.

Now assume the binary relation on each $X_{i}$ is cotransitive. Let
$f,g\in\prod_{i\in\cup\mathscr{L},\mathfrak{g}}X_{i}$. Suppose for
all $h,h'\in\prod_{i,\mathfrak{g}}X_{i}$, $h<f$ implies $h<g$,
and $h'<g$ implies $h'<f$. To show for all $i\in\cup\mathscr{L}$,
$f\left(i\right)=g\left(i\right)$, we use induction on $i$. Suppose
for all $j<i$, $f\left(j\right)=g\left(j\right)$. If $g\left(i\right)<f\left(i\right)$,
then $g<f$, so $g<g$, which is false by asymmetry. Hence, $\neg g\left(i\right)<f\left(i\right)$,
so $f\left(i\right)\leq g\left(i\right)$ by Theorem \ref{thm: 1}.
Similarly, $\neg f\left(i\right)<g\left(i\right)$, so $g\left(i\right)\leq f\left(i\right)$.
Hence, $f\left(i\right)=g\left(i\right)$, by positive antisymmetry
on $X_{i}$.\end{proof}
\begin{cor}
For any ordered set $X$, the set $X^{\mathbb{N}}$ with the lexicographic
order is a generalized ordered set. \label{cor: 23}\end{cor}
\begin{proof}
Let $X$ be an ordered set. Recall that any ordered set is a generalized
ordered set, by Theorem \ref{thm: 6}, so $X$ is a generalized ordered
set. Let $\mathscr{L}$ be the collection $\left\{ \left\{ 0,\ldots,n\right\} :n\in\mathbb{N}\right\} $.
By Theorem \ref{thm: 23}, for each finite subcollection $\left\{ A_{0},\ldots,A_{m}\right\} $
of $\mathscr{L}$, $\prod_{i\in\cup_{j=1}^{m}A_{j}}X_{i}$ is a generalized
ordered set, where $X_{i}=X$ for each $i$. Since $\cup_{n\in\mathbb{N}}\left\{ 0,\ldots,n\right\} =\mathbb{N}$
and since the binary relation $<$ on $X$ is cotransitive, it follows
$X^{\mathbb{N}}$ is a generalized ordered set by Theorem \ref{thm: 18}.
\end{proof}
Another binary relation%
\footnote{This definition is due to Michael Shulman.%
} can be defined on the cartesian product of two generalized ordered
sets. This binary relation seems to have the same peculiarity as the
lexicographic order on the cartesian product of ordered sets%
\footnote{See the first subsection of this section.%
}, namely it is known to be a generalized order when the first component
of the product satisfies the discreteness condition, that is, for
all $x,y$, $x<y$, $x=y$, or $y<x$. One question is whether that
discreteness is needed. We use discreteness to prove that relation
is transitive for the product (see the proof of Theorem \ref{thm: 29}).
Another question is whether there is any constructive proof that this
relation is transitive on the cartesian product of any two generalized
ordered sets. We have Theorem \ref{thm: 27} that suggests there may
be no such proof. Furthermore, we show that this relation is exactly
the lexicographic order when discreteness holds for the first component
(see Theorem \ref{thm: 19}). Here are the details.

For any sets $X$ and $Y$ with binary relations $<$, the \textit{weak
lexicographic order} on the cartesian product $X\times Y$ is $\left(x,y\right)<_{w}\left(x',y'\right)$
if 
\begin{enumerate}
\item $x\leq x'$;
\item $x=x'$ implies $y<y'$;
\item $\neg x=x'$ implies $x<x'$.\end{enumerate}
\begin{thm}
Let $X,Y$ be sets with asymmetric and positively antisymmetric binary
relations $<$. Then the weak lexicographic order on $X\times Y$
is asymmetric and positively antisymmetric. \label{thm: 27-1}\end{thm}
\begin{proof}
Suppose $\left(x,y\right)<_{w}\left(x',y'\right)$ and $\left(x',y'\right)<_{w}\left(x,y\right)$.
Then $x\leq x'$ and $x'\leq x$, so $x=x'$ by positive antisymmetry.
Hence, $y<y'$ and $y'<y$, which is false by asymmetry.

For positive antisymmetry, suppose the following: for all $\left(r,s\right),\left(u,v\right)\in X\times Y$,
\[
\left\{ \begin{array}{c}
\left(r,s\right)<_{w}\left(x,y\right)\textrm{ implies }\left(r,s\right)<_{w}\left(x',y'\right),\textrm{ and }\\
\left(x',y'\right)<_{w}\left(r,s\right)\textrm{ implies }\left(x,y\right)<_{w}\left(r,s\right)
\end{array}\right\} (1)
\]
\[
\left\{ \begin{array}{c}
\left(u,v\right)<_{w}\left(x',y'\right)\textrm{ implies }\left(u,v\right)<_{w}\left(x,y\right),\textrm{ and }\\
\left(x,y\right)<_{w}\left(u,v\right)\textrm{ implies }\left(x',y'\right)<_{w}\left(u,v\right)
\end{array}\right\} (2).
\]
For all $z\in X$, if $z<x$, then $\left(z,y'\right)<_{w}\left(x,y\right)$,
so $\left(z,y'\right)<_{w}\left(x',y'\right)$ by (1). Note that $\neg z=x'$
because $z=x'$ implies $y'<y'$, which is false, so $z<x'$. Now,
if $x'<z$, then $\left(x',y'\right)<_{w}\left(z,y\right)$, so $\left(x,y\right)<_{w}\left(z,y\right)$
by (1), implying $x<z$ since $\neg x=z$. Hence, $x\leq x'$. Similarly,
$x'\leq x$. Therefore, $x=x'$. For all $z\in Y$, if $z<y$, then
$\left(x,z\right)<_{w}\left(x,y\right)$, so $\left(x,z\right)<_{w}\left(x',y'\right)$.
Since $x=x'$, it follows $z<y'$. Now, if $y'<z$, then $\left(x',y'\right)<_{w}\left(x',z\right)$,
so $\left(x,y\right)<_{w}\left(x',z\right)$, implying $y<z$ since
$x=x'$. Hence, $y\leq y'$. Similarly, $y'\leq y$. Therefore, $y=y'$.\end{proof}
\begin{thm}
The statement ``for any set $X$ with an asymmetric binary relation
$<$ and for all $x,x'\in X$, $\neg x=x'$ and $x\leq x'$ imply
$x<x'$'' implies the law of double negation%
\footnote{The law of double negation says, for all proposition $P$, if $\neg\neg P$,
then $P$.%
}.\label{thm: 27}\end{thm}
\begin{proof}
Let $P$ be any proposition. Let $X=\left\{ a,b\right\} $ with exactly
$a<b$ if $P$ and $a=b$ if $\neg P$. Suppose $\neg\neg P$. Then
$\neg a=b$. Note that $a\leq b$, trivially. Hence, $a<b$. Therefore,
$P$.\end{proof}
\begin{lem}
Let $X$ be a set with an asymmetric binary relation $<$. If $X$
is discrete, then, for all $x,x'\in X$, $\neg x=x'$ and $x\leq x'$
imply $x<x'$. \label{lem: 16}\end{lem}
\begin{proof}
Since $x\leq x'$, it follows $\neg x'<x$, by Theorem \ref{thm: 3}.
Hence, $x<x'$, since $X$ is discrete.\end{proof}
\begin{thm}
Let $X$ be a discrete generalized ordered set and $Y$ any generalized
ordered set. Then $X\times Y$, with the weak lexicographic order,
is a generalized ordered set. \label{thm: 29}\end{thm}
\begin{proof}
To prove $<_{w}$ is transitive, suppose $\left(x,y\right)<_{w}\left(x',y'\right)<_{w}\left(x'',y''\right)$.
Note that $x\leq x''$ by Lemma \ref{lem: 1}(1). Now suppose $x=x''$.
Since $x\leq x'$, it follows $x''\leq x'$. Since $\left(x',y'\right)<_{w}\left(x'',y''\right)$,
it follows $x'\leq x''$. Hence, $x'=x''$ by positive antisymmetry,
so $x=x'=x''$. Thus $y<y'<y''$, implying $y<y''$. Finally, suppose
$\neg x=x''$. Since $x\leq x''$, it follows $x<x''$ by Lemma \ref{lem: 16}.
Therefore, $\left(x,y\right)<_{w}\left(x'',y''\right)$.

Asymmetry and positive antisymmetry of $<_{w}$ follow from Theorem
\ref{thm: 27-1}.\end{proof}
\begin{thm}
Let $X$ be a discrete generalized ordered set and $Y$ be any generalized
ordered set. Then $\left(x,y\right)<\left(x',y'\right)$ if and only
if $\left(x,y\right)<_{w}\left(x',y'\right)$. \label{thm: 19}\end{thm}
\begin{proof}
The forward direction is trivial. Now suppose $\left(x,y\right)<_{w}\left(x',y'\right)$.
Since $X$ is discrete, $x<x'$, $x=x'$, or $x'<x$. Note that $x=x'$
implies $y<y'$. Also, observe that $\neg x'<x$ since $x\leq x'$
implies $\neg x'<x$ by Theorem \ref{thm: 3}. Hence, $\left(x,y\right)<\left(x',y'\right)$.
\end{proof}
Two generalized ordered sets are \textit{isomorphic} if there is an
onto embedding between them, and such an embedding is called an \textit{isomorphism}.
\begin{thm}
Let $X$ be a discrete generalized ordered set and $Y$ be any generalized
ordered set. Then $\left(X\times Y,<\right)\cong\left(X\times Y,<_{w}\right)$.
\label{thm: 25}\end{thm}
\begin{proof}
The isomorphism is the identity function on $X\times Y$. It is an
embedding by Theorem \ref{thm: 19}.
\end{proof}

\section{Further Developments }

\subsection{An Excursion into Type Theory}

The notion of a \textit{set}%
\footnote{Roughly speaking, a set is defined as a special kind of a type.%
}\textit{ }can be generalized to the notion of a \textit{type} \cite{key-18,key-22,key-16}.
Properties that we have introduced earlier (e.g., asymmetry, transitivity)
can be extended to types. In what follows, we will show how this can
be done. To acquire some familiarity with type theory, these sources
can be consulted \cite{key-13-1,key-18,key-22,key-8-1} 

Before we get into more details, we introduce some basics. We will
assume the existence of these types: the type with no element, the
type with exactly one element, the type of natural numbers, the type
of propositions, denoted $\mathsf{Prop}$, and a universe, which has
types as elements. A universe is closed under certain type formers:
given types $A,B$ in a universe $\mathcal{U}$,
\begin{enumerate}
\item the type of functions, denoted by $A\rightarrow B$, is in $\mathcal{U}$;
\item the type of pairs $\left(a,b\right)$ with $a$ in $A$ and $b$ in
$B$, which is denoted by $A\times B$ , is in $\mathcal{U}$;
\item they type $A+B$, which consists of two functions $\mathsf{inl}:A\rightarrow A+B$
and $\mathsf{inr}:B\rightarrow A+B$, is in $\mathcal{U}$. 
\item the type $\left|A\right|$, which is a proposition, is in $\mathcal{U}$.
\end{enumerate}
A family of types indexed by another type $A$ is a function $B$
from $A$ to the universe $\mathcal{U}$, so each $B\left(a\right)$
is a type. Given a family of types $B:A\rightarrow\mathcal{U}$, one
can construct two other types. These are the types of dependent functions,
denoted by $\Pi_{x:A}B\left(x\right)$, where each function $f$ has
the property that $f\left(a\right)$ is in $B\left(a\right)$ for
each $a$ in $A$; and the type of dependent pairs, written as $\Sigma_{a:A}B\left(a\right)$,
where for each pair $\left(a,b\right)$, the element $a$ is in $A$
and $b$ is in $B\left(a\right)$. 

We need to introduce more concepts before we can implement our plan.
Recall we have mentioned the type of propositions, so each proposition
is a type. For instance, a theorem is a type, and a proof is an element
of that theorem. In particular, the logic used for reasoning can be
translated in terms of types under the so-called Curry-Howard correspondence.
Here are the details. Given propositions $P,Q$, 
\begin{enumerate}
\item ``$\neg P$'' is the type $P\rightarrow\mathbf{0}$, where $\mathbf{0}$
is the type with no element;
\item ``$P$ and $Q$'' is the type $P\times Q$;
\item ``$P$ or $Q$'' is the type $\left|P+Q\right|$;
\item ``if $P$, then $Q$'' is the type $P\rightarrow Q$;
\item for any type $A$, ``for all $x:A$, $Q\left(x\right)$'' is the
type $\Pi_{x:A}Q\left(x\right)$;
\item for any type $A$, ``there exists $x:A$, $Q\left(x\right)$'' is
the type $\left|\Sigma_{x:A}Q\left(x\right)\right|$.
\end{enumerate}
Before talking about binary relations, we need to mention that study
of types with binary relations is nothing new. The real numbers \cite{key-18,key-10},
defined both as Dedekind cuts and Cauchy sequences, and the surreal
numbers \cite{key-18} are examples of types with binary relations.
A binary relation on a type $A$ is a function from $A\times A$ to
the type of propositions. Instead of writing a binary relation as
a function $\mathsf{rel}:A\times A\rightarrow\mathsf{Prop}$, we will
follow the usual practice of \textit{currying }to write a binary relation
as $\mathsf{rel}:A\rightarrow A\rightarrow\mathsf{Prop}$, which is
a function from $A$ to $A\rightarrow\mathsf{Prop}$. 

To differentiate the asymmetric%
\footnote{See the third section for the definition of asymmetry.%
} property for sets from what we will define for types, we will use
the letter ``t'' as a prefix. We will, though, use the notation
$<$ for the binary relation and will write ``$x<y$'' to mean $<\left(x,y\right)$.
Since we need a type and a binary relation to define t-asymmetry,
we expect the definition to depend on two parameters. Plainly, a binary
relation $<$ on a type $A$ is \textit{t-asymmetric} if, for all
$x,y$ in $A$, $x<y$ implies that $y<x$ is false. Under the Curry-Howard
correspondence, we can write t-asymmetry as 
\[
\textrm{t-asym}\left(A,<\right):=\Pi_{x,y:A}x<y\rightarrow\left(y<x\rightarrow\mathbf{0}\right).
\]
The other properties can be defined in a similar way, so a \textit{generalized
ordered type} $A$ can be defined as the type 
\[
\textrm{genOrd}\left(A,<\right):=\textrm{t-asym}\left(A,<\right)\times\textrm{t-trans}\left(A,<\right)\times\textrm{t-posAnt}\left(A,<\right).
\]

\subsection{Problems}

\subsubsection{Isomorphic theories}

The notion of a set can be captured in type theory. There are several
approaches, some of which are translation of Zermelo-Fraenkel set
theory in type theory \cite{key-15}, definition of a set as a special
kind of type \cite{key-18}, and definition%
\footnote{This definition of a set as a pair is sometimes called a \textit{setoid}
or a \textit{Bishop set}.%
} of a set as a pair made up of a type and a binary relation on that
type \cite{key-16}. Whichever approach is used, the notion of a generalized
ordered set can be translated in type theory, so a theory of generalized
ordered sets can be developed inside type theory. Futhermore, we have
shown from the previous subsection how the notion of a generalized
ordered type can be defined in type theory, so a theory of generalized
ordered types can also be developed. We designate the former theory
as $\mathsf{Th}_{genOrd,set}$ and the latter as $\mathsf{Th}_{genOrd,type}$.
Is there a constructive proof that these two theories are essentially
the same, in the sense that for each theorem $T$ of $\mathsf{Th}_{genOrd,set}$
there is a theorem $T'$ of $\mathsf{Th}_{genOrd,type}$ such that
$T$ and $T'$ are equivalent in the ambient type theory, and vice
versa?

\subsubsection{Unique partial order }

Different partial orders can be defined on a set with a binary relation
$<$. For instance, $\leq_{N}$, $\leq_{P}$, and the relation $x\leq y$
if $x<y$ or $x=y$ are all partial orders (see Sections 3 \& 4),
but they are not all the same: on the real numbers, there is no constructive
proof that $\leq_{N}$ is the same as $\leq$ (see 3.1). However,
$\leq_{N}$ and $\leq_{P}$ are the same if $<$ is asymmetric and
cotransitive (Theorems \ref{thm: 1} \& \ref{thm: 3}). A natural
question is whether it is possible to impose certain conditions on
the relation $<$ such that any partial order defined from $<$ is
unique.

\subsubsection{Generalized $\mathbb{N}$-ordered Sets}

Given any set $X$, we can define an $\mathbb{N}$-ary relation on
$X$ as any subset of $X^{\mathbb{N}}$, where $X^{\mathbb{N}}$ is
the set of all functions from the natural numbers $\mathbb{N}$ to
$X$. We can follow a similar convention to a binary relation when
writing that an infinite sequence of elements of $X$ are related.
We denote an $\mathbb{N}$-ary relation as $<_{\mathbb{N}}$, and
we write $<_{\mathbb{N}}\left(x_{1},x_{2},\ldots\right)$ to mean
that the terms of the infinite sequence $x_{1},x_{2},\ldots$ are
related under $<_{\mathbb{N}}$. A question is what a natural definition
of asymmetry%
\footnote{See Section 3.%
}, transitivity%
\footnote{See Section 3.%
}, and positive antisymmetry%
\footnote{See Section 3.%
} is for any $\mathbb{N}$-ary relation.

Furthermore, for each $k\in\mathbb{N}$, any $\mathbb{N}$-ary relation
on $X$ yields a natural $k$-ary relation on $X$, which is the restriction
of the $\mathbb{N}$-ary relation on the first $k$ terms of each
$<_{\mathbb{N}}\left(x_{1},x_{2},\ldots\right)$. The identity function
on $X$ naturally preserves the $\mathbb{N}$-ary relation to the
$k$-ary relation, in the sense that if $<_{\mathbb{N}}\left(x_{1},x_{2},\ldots\right)$
then $<_{k}\left(x_{1},x_{2},\ldots,x_{k}\right)$ for each $k\in\mathbb{N}$.
We denote by $\textrm{id}_{k}$ that natural order-preserving identity
function from $\left(X,<_{\mathbb{N}}\right)$ to $\left(X,<_{k}\right)$.
A question is whether there is a natural universal property for $\mathbb{N}$-ary
relations in this sense: for any $\mathbb{N}$-ary relation $<'_{\mathbb{N}}$
on $X$ and any sequence of order-preserving functions $f_{k}$ from
$\left(X,<'_{\mathbb{N}}\right)$ to $\left(X,<_{k}\right)$, there
is a unique order-preserving function $g$ from $\left(X,<'_{\mathbb{N}}\right)$
to $\left(X,<_{\mathbb{N}}\right)$ such that $f_{k}=\textrm{id}_{k}\circ g$
for each $k$.

\bigskip{}

\bigskip{}

\bigskip{}

\bigskip{}

\thanks{The author thanks Michael Shulman and Mart\'in Escard\'o for bringing
$\mathsf{Prop}$ to his attention.}

\end{document}